\documentclass[12pt]{article}
\usepackage{latexsym, amssymb, amsthm}
\usepackage{dsfont}
\usepackage{color}
\usepackage{enumerate}

\DeclareMathAlphabet\gothic{U}{euf}{m}{n}

\makeatletter
\def\eqnarray{\stepcounter{equation}\let\@currentlabel=\theequation
\global\@eqnswtrue
\tabskip\@centering\let\\=\@eqncr
$$\halign to \displaywidth\bgroup\hfil\global\@eqcnt\z@
  $\displaystyle\tabskip\z@{##}$&\global\@eqcnt\@ne
  \hfil$\displaystyle{{}##{}}$\hfil
  &\global\@eqcnt\tw@ $\displaystyle{##}$\hfil
  \tabskip\@centering&\llap{##}\tabskip\z@\cr}

\def\endeqnarray{\@@eqncr\egroup
      \global\advance\c@equation\m@ne$$\global\@ignoretrue}

\def\@yeqncr{\@ifnextchar [{\@xeqncr}{\@xeqncr[5pt]}}
\makeatother

\begin{document}
\bibliographystyle{tom}

\newtheorem{lemma}{Lemma}[section]
\newtheorem{thm}[lemma]{Theorem}
\newtheorem{cor}[lemma]{Corollary}
\newtheorem{prop}[lemma]{Proposition}
\newtheorem{ddefinition}[lemma]{Definition}

\theoremstyle{definition}

\newtheorem{remark}[lemma]{Remark}
\newtheorem{exam}[lemma]{Example}
\newtheorem{hyp}[lemma]{Hypothesis}

\newcommand{\gota}{\gothic{a}}
\newcommand{\gotb}{\gothic{b}}
\newcommand{\gotc}{\gothic{c}}
\newcommand{\gote}{\gothic{e}}
\newcommand{\gotf}{\gothic{f}}
\newcommand{\gotg}{\gothic{g}}
\newcommand{\gothh}{\gothic{h}}
\newcommand{\gotk}{\gothic{k}}
\newcommand{\gotl}{\gothic{l}}
\newcommand{\gotm}{\gothic{m}}
\newcommand{\gotn}{\gothic{n}}
\newcommand{\gotp}{\gothic{p}}
\newcommand{\gotq}{\gothic{q}}
\newcommand{\gotr}{\gothic{r}}
\newcommand{\gots}{\gothic{s}}
\newcommand{\gott}{\gothic{t}}
\newcommand{\gotu}{\gothic{u}}
\newcommand{\gotv}{\gothic{v}}
\newcommand{\gotw}{\gothic{w}}
\newcommand{\gotz}{\gothic{z}}
\newcommand{\gotA}{\gothic{A}}
\newcommand{\gotB}{\gothic{B}}
\newcommand{\gotG}{\gothic{G}}
\newcommand{\gotL}{\gothic{L}}
\newcommand{\gotS}{\gothic{S}}
\newcommand{\gotT}{\gothic{T}}

\newcounter{teller}
\renewcommand{\theteller}{(\alph{teller})}
\newenvironment{tabel}{\begin{list}%
{\rm  (\alph{teller})\hfill}{\usecounter{teller} \leftmargin=1.1cm
\labelwidth=1.1cm \labelsep=0cm \parsep=0cm}
                      }{\end{list}}

\newcounter{tellerr}
\renewcommand{\thetellerr}{(\roman{tellerr})}
\newenvironment{tabeleq}{\begin{list}%
{\rm  (\roman{tellerr})\hfill}{\usecounter{tellerr} \leftmargin=1.1cm
\labelwidth=1.1cm \labelsep=0cm \parsep=0cm}
                         }{\end{list}}

\newcounter{tellerrr}
\renewcommand{\thetellerrr}{(\Roman{tellerrr})}
\newenvironment{tabelR}{\begin{list}%
{\rm  (\Roman{tellerrr})\hfill}{\usecounter{tellerrr} \leftmargin=1.1cm
\labelwidth=1.1cm \labelsep=0cm \parsep=0cm}
                         }{\end{list}}

\newcounter{proofstep}
\newcommand{\nextstep}{\refstepcounter{proofstep}\vertspace \par 
          \noindent{\bf Step \theproofstep.} \hspace{5pt}}
\newcommand{\firststep}{\setcounter{proofstep}{0}\nextstep}

\newcommand{\Ni}{\mathds{N}}
\newcommand{\Qi}{\mathds{Q}}
\newcommand{\Ri}{\mathds{R}}
\newcommand{\Ci}{\mathds{C}}
\newcommand{\Ti}{\mathds{T}}
\newcommand{\Zi}{\mathds{Z}}
\newcommand{\Fi}{\mathds{F}}

\renewcommand{\proofname}{{\bf Proof}}

\newcommand{\vertspace}{\vskip10.0pt plus 4.0pt minus 6.0pt}

\newcommand{\simh}{{\stackrel{{\rm cap}}{\sim}}}
\newcommand{\ad}{{\mathop{\rm ad}}}
\newcommand{\Ad}{{\mathop{\rm Ad}}}
\newcommand{\alg}{{\mathop{\rm alg}}}
\newcommand{\clalg}{{\mathop{\overline{\rm alg}}}}
\newcommand{\Aut}{\mathop{\rm Aut}}
\newcommand{\arccot}{\mathop{\rm arccot}}
\newcommand{\capp}{{\mathop{\rm cap}}}
\newcommand{\rcapp}{{\mathop{\rm rcap}}}
\newcommand{\diam}{\mathop{\rm diam}}
\newcommand{\divv}{\mathop{\rm div}}
\newcommand{\dom}{\mathop{\rm dom}}
\newcommand{\codim}{\mathop{\rm codim}}
\newcommand{\RRe}{\mathop{\rm Re}}
\newcommand{\IIm}{\mathop{\rm Im}}
\newcommand{\tr}{{\mathop{\rm tr \,}}}
\newcommand{\Tr}{{\mathop{\rm Tr \,}}}
\newcommand{\Vol}{{\mathop{\rm Vol}}}
\newcommand{\card}{{\mathop{\rm card}}}
\newcommand{\rank}{\mathop{\rm rank}}
\newcommand{\supp}{\mathop{\rm supp}}
\newcommand{\sgn}{\mathop{\rm sign}}
\newcommand{\essinf}{\mathop{\rm ess\,inf}}
\newcommand{\esssup}{\mathop{\rm ess\,sup}}
\newcommand{\Int}{\mathop{\rm Int}}
\newcommand{\lcm}{\mathop{\rm lcm}}
\newcommand{\loc}{{\rm loc}}
\newcommand{\HS}{{\rm HS}}
\newcommand{\Trn}{{\rm Tr}}
\newcommand{\n}{{\rm N}}
\newcommand{\WOT}{{\rm WOT}}

\newcommand{\at}{@}

\newcommand{\mod}{\mathop{\rm mod}}
\newcommand{\spann}{\mathop{\rm span}}
\newcommand{\one}{\mathds{1}}

\hyphenation{groups}
\hyphenation{unitary}

\newcommand{\tfrac}[2]{{\textstyle \frac{#1}{#2}}}

\newcommand{\ca}{{\cal A}}
\newcommand{\cb}{{\cal B}}
\newcommand{\cc}{{\cal C}}
\newcommand{\cd}{{\cal D}}
\newcommand{\ce}{{\cal E}}
\newcommand{\cf}{{\cal F}}
\newcommand{\ch}{{\cal H}}
\newcommand{\chs}{{\cal HS}}
\newcommand{\ci}{{\cal I}}
\newcommand{\ck}{{\cal K}}
\newcommand{\cl}{{\cal L}}
\newcommand{\cm}{{\cal M}}
\newcommand{\cn}{{\cal N}}
\newcommand{\co}{{\cal O}}
\newcommand{\cp}{{\cal P}}
\newcommand{\cs}{{\cal S}}
\newcommand{\ct}{{\cal T}}
\newcommand{\cx}{{\cal X}}
\newcommand{\cy}{{\cal Y}}
\newcommand{\cz}{{\cal Z}}

\newlength{\hightcharacter}
\newlength{\widthcharacter}
\newcommand{\covsup}[1]{\settowidth{\widthcharacter}{$#1$}\addtolength{\widthcharacter}{-0.15em}\settoheight{\hightcharacter}{$#1$}\addtolength{\hightcharacter}{0.1ex}#1\raisebox{\hightcharacter}[0pt][0pt]{\makebox[0pt]{\hspace{-\widthcharacter}$\scriptstyle\circ$}}}
\newcommand{\cov}[1]{\settowidth{\widthcharacter}{$#1$}\addtolength{\widthcharacter}{-0.15em}\settoheight{\hightcharacter}{$#1$}\addtolength{\hightcharacter}{0.1ex}#1\raisebox{\hightcharacter}{\makebox[0pt]{\hspace{-\widthcharacter}$\scriptstyle\circ$}}}
\newcommand{\scov}[1]{\settowidth{\widthcharacter}{$#1$}\addtolength{\widthcharacter}{-0.15em}\settoheight{\hightcharacter}{$#1$}\addtolength{\hightcharacter}{0.1ex}#1\raisebox{0.7\hightcharacter}{\makebox[0pt]{\hspace{-\widthcharacter}$\scriptstyle\circ$}}}

\thispagestyle{empty}

\vspace*{1cm}
\begin{center}
{\Large\bf The diamagnetic inequality for the  \\[2mm]
Dirichlet-to-Neumann operator 
\\[10mm]
\large A.F.M. ter Elst$^1$ and E.M. Ouhabaz$^2$}

\end{center}

\vspace{5mm}

\begin{center}
{\bf Abstract}
\end{center}

\begin{list}{}{\leftmargin=1.8cm \rightmargin=1.8cm \listparindent=10mm 
   \parsep=0pt}
\item
Let $\Omega$ be a bounded domain in $\Ri^d$ with Lipschitz boundary $\Gamma$. 
We define the Dirichlet-to-Neumann operator $\mathcal N$ on $L_2(\Gamma)$ 
associated with a second order elliptic operator 
$A = -\sum_{k,j=1}^d \partial_k (c_{kl} \, \partial_l) 
   + \sum_{k=1}^d b_k \, \partial_k - \partial_k( c_k \cdot) + a_0$. 
We prove a criterion  for invariance of a closed convex set under the 
action of the semigroup of $\cn$. 
Roughly speaking, it says that if the 
semigroup generated by $-A$, endowed with Neumann boundary conditions, 
leaves invariant a closed convex set of $L_2(\Omega)$, then the `trace' of this 
convex set is invariant for the semigroup of~$\cn$. 
We use this invariance to prove a criterion for the domination of  semigroups of two 
Dirichlet-to-Neumann operators. We apply this criterion to prove the diamagnetic inequality  for such operators on $L_2(\Gamma)$. 

\end{list}

\vspace{5mm}
\noindent
April 2020

\vspace{5mm}

\vspace{5mm}

\vspace{15mm}

\noindent
{\bf Home institutions:}    \\[3mm]
\begin{tabular}{@{}cl@{\hspace{10mm}}cl}
1. & Department of Mathematics  & 
  2. & Institut de Math\'ematiques de Bordeaux \\
& University of Auckland   & 
  & Universit\'e de Bordeaux, UMR 5251,  \\
& Private bag 92019 & 
  &351, Cours de la Lib\'eration  \\
& Auckland 1142 & 
  &  33405 Talence \\
& New Zealand  & 
  & France\\
  & terelst@math.auckland.ac.nz&
 & Elmaati.Ouhabaz@math.u-bordeaux.fr \\[8mm]
\end{tabular}

\newpage

\section{Introduction} \label{Auckland}

The well known diamagnetic inequality states that the semigroup associated with a 
Schr\"o\-din\-ger operator with a magnetic field is pointwise bounded by the 
free semigroup  of the Laplacian. 
More precisely,  let $\vec{a} = (a_1,\ldots,a_d)$ be such that each $a_k$ is  
real-valued and  locally in $L_2(\Ri^d)$.
Set $H(\vec{a}) = (\nabla-i \vec{a})^*(\nabla-i \vec{a})$.
Then the corresponding semigroup $(e^{-t H(\vec{a})})_{t \geq 0}$ satisfies
\[
| e^{-t H(\vec{a})} f | \le e^{t \Delta} | f |
\]
for all $t > 0$ and $f \in L_2(\Ri^d)$.
The same result holds in presence of a real-valued potential~$V$, i.e.,  
with operators  $H(\vec{a})  + V$ and $-\Delta + V$.\\
The diamagnetic inequality plays an important role in spectral theory of Schr\"odinger operators with magnetic potential.
We refer to \cite{Sim82} and references there.
 
 The main objective of the present paper is to prove a similar result for the 
Dirichlet-to-Neumann operator with magnetic field on the boundary $\Gamma$ of a Lipschitz domain
$\Omega$ in $\Ri^d$.
 In its most simplest case, the diamagnetic inequality we prove says that for all 
$\vec{a} \in (L_\infty(\Omega, \Ri))^d$, the solutions of the two problems
\[ 
\left\{ \begin{array}{r@{}c@{}ll}
\partial_t  \Tr u +  (\partial_\nu - i \vec{a} \cdot \nu ) u & {} = {} & 0 & \mbox{on } (0, \infty) \times \Gamma\\[5pt]
(\nabla - i \vec{a})^*(\nabla -i \vec{a}) u & = & 0 & \mbox{on }  (0, \infty) \times \Omega\\[5pt]
\Tr u & = & \varphi 
 \end{array}
\right.
 \]
and 
\[ 
\left\{ \begin{array}{r@{}c@{}ll}
\partial_t  \Tr v +  \partial_\nu v & {} = {} & 0 & \mbox{on } (0, \infty) \times \Gamma\\[5pt]
\Delta v & = & 0  & \mbox{on } (0, \infty) \times \Omega\\[5pt]
\Tr v & = & |\varphi| 
 \end{array}
\right.
 \]
satisfy
\[ 
| u(t,x) | \le v(t,x) \mbox{ for a.e.\ } (t,x) \in (0, \infty) \times \Gamma.
\]
Here $\partial_\nu$ is the normal derivative and $\nu$ is the outer normal vector to $\Omega$.
 We prove more in the sense that we are able to deal with variable and non-symmetric coefficients.
To be more  precise, we consider  $c_{kl}, b_k, c_k, a_0  \in L_\infty (\Omega, \Ri)$ 
for all $k,l \in \{ 1,\ldots,d \} $ with $(c_{kl})$ satisfying the usual ellipticity condition.
For all $\vec{a} \in (L_\infty(\Omega, \Ri))^d$ as above we consider the magnetic 
Dirichlet-to-Neumann operator $\cn(\vec{a})$ defined as follows.
If $\varphi \in H^{1/2}(\Gamma)$, we solve first 
\begin{eqnarray*}
-\sum_{k,l= 1}^d (\partial_k - i a_k)  \Big( c_{kl} \, (\partial_l -i a_l)\,  u \Big)
     + \sum_{k=1}^d b_k \, (\partial_k - ia_k) u 
    - (\partial_k - i a_k)  (c_k\,  u)  + a_0\, u & = & 0 
      \quad \mbox{on } \Omega,  \\[0pt]
\Tr u & = & \varphi 
\end{eqnarray*}
with $u \in W^{1,2}(\Omega)$ and then define $\cn(\vec{a}) \varphi$ as the conormal derivative
(when it exists as an element of $L_2(\Gamma)$).
Formally, 
\[
\cn(\vec{a}) \varphi = \sum_{k,l=1}^d \nu_k \, \Tr(c_{kl} \, \partial_l u)  
   - i \sum_{k,l=1}^d \nu_k \Tr( c_{kl} \, a_l \,  u) 
   + \sum_{k=1}^d \nu_k \, \Tr(c_k\, u)
.  \]
If $(S_{\vec{a}} (t))_{t\ge 0}$ denotes the semigroup generated by $-\cn(\vec{a})$ on $L_2(\Gamma)$, 
then we prove (under an accretivity condition) that
\[ | S_{\vec{a}}(t) \varphi | \le S_{0}(t) |\varphi |
\]
for all $t \ge 0$ and $\varphi \in L_2(\Gamma)$.
In the symmetric case, $c_{kl} = c_{lk}$ and $b_k = c_k = 0$, we obtain as a consequence 
a trace norm estimate for the eigenvalues of $\cn(\vec{a})$ and if the coefficients 
are H\"older continuous and $\Omega$ is of class $C^{1+ \kappa}$ for some 
$\kappa > 0$ we obtain that the heat kernel of $\cn(\vec{a})$ satisfies a 
Poisson upper bound on $\Gamma$.
We also prove other results on positivity (when $\vec{a} = 0$) and 
$L_\infty$-contractivity of the corresponding semigroup. 
For example, in the symmetric case  $c_{kl} = c_{lk}$, $b_k = c_k $ and $a_0$ are all real, 
 then the semigroup $S_0$ is positive if $a_0 > - \lambda_0$, where
$\lambda_0$ is the first positive eigenvalue of the elliptic operator 
 \begin{equation}\label{eq-ellip01}
   -\sum_{k,j=1}^d \partial_k (c_{kl} \, \partial_l) 
   + \sum_{k=1}^d b_k \, \partial_k - \partial_k( c_k \cdot)
   \end{equation}
   with Dirichlet boundary conditions. In other words, the quadratic form
\begin{equation}\label{pos01}
u \mapsto \gota(u,u) 
= \sum_{k,l=1}^d  \int_\Omega c_{kl} \, (\partial_l u) \, \overline{\partial_k u} 
   + \sum_{k=1}^d  \int_\Omega b_k \, (\partial_k u) \, \overline{u} + 
      b_k \, u\, \overline{ \partial_k u} + \int_\Omega a_0 \, | u |^2 
\end{equation}
is positive on $W^{1,2}_0(\Omega)$. 
It is not clear whether this latter condition remains sufficient for positivity of the semigroup 
in the non-symmetric case.
 See 
Proposition~\ref{invC-AE-New} and Section~\ref{pos}. 

It is worth mentioning that the Dirichlet-to-Neumann operator is an important map 
which appears in many problems.
In particular, it plays a fundamental role in inverse problems such as 
the Calder\'on inverse problem.
The magnetic Dirichlet-to-Neumann operator also appears in the study of inverse problems in the presence of a magnetic field.
We refer to \cite{BC} and the references therein. 

In order to prove the diamagnetic inequality we proceed by invariance of closed convex sets for an appropriate semigroup.
This idea appeared already in \cite{Ouh95}.
Despite the fact that it is an abstract result, 
the invariance result proved in \cite{Ouh95}, however, 
does not seem to apply in an efficient way to the Dirichlet-to-Neumann operator.
The reason is that in this setting one has to deal with the harmonic lifting 
(with respect to the elliptic operator) of functions and  it is not clear how 
to describe such harmonic lifting for complicated expressions 
(see Section~\ref{Whangamomona} below).
What we do is to rely first on a version from \cite{AE2} of the invariance 
criterion of \cite{Ouh95} and then prove  new criteria for invariance of 
closed convex sets which make a bridge between invariance on $L_2(\Gamma)$ for 
the Dirichlet-to-Neumann semigroup and invariance on $L_2(\Omega)$ for the 
semigroup of the elliptic operator with Neumann boundary conditions.
The latter is easier to handle. 
The result is efficient when dealing with the Dirichlet-to-Neumann operator.
The diamagnetic inequality is obtained from a domination criterion which is 
obtained by checking the invariance of the convex set
$\{ (\varphi, \psi) \in L_2(\Gamma) \times L_2(\Gamma) : |\varphi | \le \psi \}$ 
for the semigroup $\left( { \begin{array}{cc}
S_{\vec{a}}(t) & 0 \\
0 & S_{0}(t)\\
\end{array} } \right)_{t\ge0}$.

 \section{Background material}\label{Stratford}
 
 The aim of this section is to recall some well known material on sesquilinear forms and make precise several  notations which will be used throughout the paper.
 
 Let $\widetilde{H}$ be a Hilbert space with scalar product $(\cdot,\cdot)_{\widetilde{H}}$ and  associated norm 
$\| \cdot\|_{\widetilde{H}}$.
Given another Hilbert space $V$ which is continuously and densely embedded into $\widetilde H$. 
Let 
 \[ \gota \colon V \times V \to \Ci
 \]
 be a sesquilinear form.
We assume that $\gota$ is {\bf continuous} and {\bf quasi-coercive}.
This means respectively that there exist constants $M \ge 0$, $\mu > 0$ and 
$\omega \in \Ri$ such that 
\begin{eqnarray*}
| \gota(u,v)| & \le & M \, \| u \|_V \, \| v \|_V \mbox{ and}   \\
\RRe \gota(u,u) + \omega \, \| u \|_{\widetilde{H}}^2 & \ge & \mu \, \| u \|_V^2
 \end{eqnarray*}
for all $u, v \in V$.
It then follows that $\gota$ is a closed sectorial form and hence one can associate an 
operator $\widetilde{A}$ on $\widetilde H$ such that for all $(u,f) \in \widetilde H \times \widetilde H$ one has
\[
u \in D(\widetilde A)
\quad \mbox{and} \quad
\widetilde A u = f
\]
if and only if
\[
u \in V 
\quad \mbox{and} \quad
\gota(u, v) = (f, v)_{\widetilde{H}}
\mbox{ for all } v \in V
.  \]
It is a standard fact that $\widetilde{A}$ is a densely defined 
(quasi-)sectorial operator and $-\widetilde{A}$ generates a holomorphic semigroup
$\widetilde{S} = (\widetilde{S}(t))_{t\ge0}$ on $\widetilde{H}$.
See, e.g., \cite{Kat1} or \cite{Ouh5}.
 
Let now $H$ be another Hilbert space and $j \colon V \to H$ be a linear continuous map with dense range.
Suppose that the form 
$\gota \colon V \times V \to \Ci$ is continuous.
Following \cite{AE2}, we say that 
 $\gota$ is {\bf $j$-elliptic} if there exist constants $\omega \in \Ri$ and $\mu > 0$ such that
\[
 \RRe \gota(u,u) + \omega \, \|j(u) \|_H^2 \ge \mu \,  \| u \|_V^2
\]
for all $u \in V$.
In this case, there exists an operator $A$, called the operator associated with~$(\gota,j)$, 
defined as follows.
For all $(\varphi, \psi) \in H \times H$ one has  
 \[ \varphi \in D(A)
\quad \mbox{and} \quad
A \varphi = \psi  
\]
if and only if
\[ 
\mbox{there exists a } u \in V \mbox{ such that }
\left[ \begin{array}{l}
   j(u) = \varphi \mbox{ and}  \\[5pt]
   \gota(u, v) = (\psi, j(v))_H \mbox{ for all } v \in V .
       \end{array} \right.
\]
Then $A$ is well defined and $-A$ generates a holomorphic semigroup $S = (S(t))_{t\ge0}$ on $H$.
(See \cite{AE2} Theorem 2.1.)
  
 We illustrate these definitions by two important examples in which we define the 
Dirichlet-to-Neumann operator and the magnetic Dirichlet-to-Neumann operator on the 
boundary of a Lipschitz domain.
 
 \begin{exam}[The Dirichlet-to-Neumann operator] \label{xdia201}
Let $\Omega$ be a bounded Lipschitz domain in $\Ri^d$ with boundary $\Gamma$.
We denote by $\Tr  \colon W^{1,2}(\Omega) \to L_2(\Gamma)$ the trace operator.
Let $c_{kl}, b_k, c_k, a_0 \in L_\infty (\Omega, \Ci)$ for all $k, l \in \{ 1, \ldots, d \} $. 
We assume the usual ellipticity condition: there exists a constant $\mu > 0$ such that
\begin{equation}
\RRe \sum_{k,l=1}^d c_{kl}(x) \, \xi_k \, \overline{\xi_l}
\geq \mu \, |\xi|^2
\label{exdia201;5}
\end{equation}
for all $\xi \in \Ci^d$ and almost every $x \in \Omega$. 
We define the sesquilinear form $\gota \colon W^{1,2}(\Omega) \times W^{1,2}(\Omega) \to \Ci$ by
\begin{equation}\label{example1}
\gota(u,v) 
= \sum_{k,l=1}^d  \int_\Omega c_{kl} \, (\partial_l u) \, \overline{\partial_k v} 
   + \sum_{k=1}^d  \int_\Omega b_k \, (\partial_k u) \, \overline{v} + 
      c_k \, u\, \overline{ \partial_k v} + \int_\Omega a_0 \, u \, \overline{v}
.
\end{equation}
It is a basic fact that the form $\gota$ is continuous and quasi-coercive.
We denote by $A$ the operator associated with $\gota$ on $L_2(\Omega)$.
Define the operator $\mathcal A \colon W^{1,2}(\Omega) \to W^{-1,2}(\Omega)$ by
\[ 
\langle {\mathcal A} u, v \rangle_{W^{-1,2}(\Omega) \times  W^{1,2}_0(\Omega)} = \gota(u,v).
\]
Let $u \in W^{1,2}(\Omega)$ with ${\mathcal A} u \in L_2(\Omega)$ and $\psi \in L_2(\Gamma)$. 
Then we say that $u$ has {\bf weak conormal derivative} $\psi$ if 
\[ \gota(u,v) - ({\mathcal A} u, v)_{L_2(\Omega)} = (\psi, \Tr v)_{L_2(\Gamma)} 
\]
for all $v \in W^{1,2}(\Omega)$. 
Then $\psi$ is unique by the Stone--Weierstra\ss\ theorem 
and we write $\partial^{\gota}_\nu\, u = \psi$.
Formally, 
\[
\partial^\gota_\nu\,  u = \sum_{k,l=1}^d \nu_k \, \Tr(c_{kl} \, \partial_l u) +  \sum_{k=1}^d \nu_k \Tr(c_k\, u)
,  \]
where $(\nu_1, \ldots, \nu_d)$ is the outer normal vector to $\Omega$. 
Suppose now that $0$ is not in the spectrum of $\mathcal A$ endowed with Dirichlet boundary 
conditions (i.e., the form $\gota$ is taken on $V = W^{1,2}_0(\Omega)$). 
Then we say that $u \in W^{1,2}(\Omega)$ is {\bf ${\mathcal A}$-harmonic} if
\[ \gota(u,v) = 0 
\]
for all $v \in W^{1,2}_0(\Omega)$.
Since $0$ is not in the spectrum of the Dirichlet operator, 
for all $\varphi \in H^{1/2}(\Gamma)$ there exists a unique ${\mathcal A}$-harmonic
$u \in H^1(\Omega)$ such that $\Tr u = \varphi$.
We then define on $L_2(\Gamma)$ the form $\gotb \colon  H^{1/2}(\Gamma) \times H^{1/2}(\Gamma) \to \Ci$
by
\begin{equation}\label{eq-b}
\gotb(\varphi, \xi) := \gota(u,v) ,
\end{equation}
where $u, v \in W^{1,2}(\Omega)$ are ${\mathcal A}$-harmonic with $\Tr u = \varphi$ 
and $\Tr v = \xi$, respectively. 
One proves that the form
$\gotb$ is continuous, sectorial and closed.
The associated operator $\cn$ is the Dirichlet-to-Neumann operator.
  For more details see \cite{EO4} Section~2, \cite{EO6} Section~2 or \cite{ArM2}.
The operator $\cn$ is interpreted as follows. 
For all $\varphi \in H^{1/2}(\Gamma)$, one  solves the Dirichlet problem
\begin{eqnarray*}
-\sum_{k,l= 1}^d \partial_k \Big( c_{kl} \, \partial_l \,  u \Big) + \sum_{k=1}^d b_k \, \partial_k u 
    - \partial_k (c_k\,  u)  + a_0\, u & = & 0 
      \quad \mbox{weakly in } \Omega,   \\[0pt]
\Tr u & = & \varphi 
\end{eqnarray*}
with $u \in W^{1,2}(\Omega)$ and 
if $u$ has a weak conormal derivative, then $\varphi \in D(\cn)$ and 
$\cn  \varphi = \partial^\gota _\nu\, u$.
Alternatively,  let $j := \Tr$ and $H = L_2(\Gamma)$. 
Suppose in addition that $\gota$ is $j$-elliptic, 
that is, suppose that $\RRe a_0$ is large enough.
 Then one checks easily that $\cn$ is the operator associated with $(\gota, j)$. 
\end{exam}

\begin{exam}[The magnetic Dirichlet-to-Neumann operator] \label{xdia202}
Let $\vec{a} := (a_1, \ldots, a_d)$ be such that  $a_k \in L_\infty(\Omega, \Ri)$ 
for all $k \in \{ 1, \ldots, d \} $.
Set
\[
D_k := \partial_k - i a_k
\]
for all $k \in \{ 1, \ldots, d \} $.
We define as above $\gota(\vec{a}) \colon W^{1,2}(\Omega) \times W^{1,2}(\Omega) \to \Ci$ by
\begin{equation}\label{example2}
\gota(\vec{a})(u,v) 
= \sum_{k,l=1}^d  \int_\Omega c_{kl} \, (D_l u) \, \overline{D_k v} 
   + \sum_{k=1}^d  \int_\Omega b_k \, (D_k u) \, \overline{v} 
      + c_k \, u\, \overline{ D_k v} + \int_\Omega a_0 \, u \, \overline{v}.
\end{equation}
Then one can define exactly as above the associated operator 
$A(\vec{a})$ on $L_2(\Omega)$ as well as the magnetic Dirichlet-to-Neumann operator 
$\cn(\vec{a})$. 
Formally, if $u \in W^{1,2}(\Omega)$  is 
${\mathcal A}(\vec{a})$-harmonic with trace $\Tr u = \varphi$, then 
\[
\cn(\vec{a}) \varphi 
= \partial^{\gota(\vec{a})}_\nu\,  u 
= \sum_{k,l=1}^d \nu_k \, \Tr(c_{kl} \, \partial_l u)  
   - i \sum_{k,l=1}^d \nu_k \Tr(c_{kl} \, a_l \,  u) 
   + \sum_{k=1}^d \nu_k \, \Tr(c_k\, u).
\]
\end{exam}

 \section{Invariance of closed convex sets}\label{Taranaki}
 
As previously, we denote by  $\widetilde{H}$ and  $V$  two  Hilbert spaces such that 
$V$ is continuously and densely embedded into $\widetilde{H}$.
Let 
$\gota \colon  V \times V \to \Ci$ be a quasi-coercive and continuous sesquilinear form. 
We denote by $\widetilde{A}$ the corresponding operator and 
$\widetilde S = (\widetilde S(t))_{t\ge0}$ the semigroup generated by $-\widetilde{A}$ on $\widetilde{H}$. 

Let $\widetilde{\mathcal C}$ be a non-empty closed convex subset of $\widetilde{H}$ and  
$\widetilde{P} \colon  \widetilde{H} \to \widetilde{\mathcal C}$ the corresponding projection.
We recall the following invariance criterion (see \cite{Ouh95} or \cite{Ouh5} Theorem~2.2).

\begin{thm} \label{invC-Ou}
The following conditions are equivalent.
\begin{tabeleq}
\item \label{invC-Ou-1}
The semigroup $\widetilde S$ leaves invariant $\widetilde{\mathcal C}$, that is, 
$\widetilde S(t) \widetilde{\mathcal C} \subset \widetilde{\mathcal C}$ for all $t \ge 0$.
\item \label{invC-Ou-2}
$\widetilde{P}V \subset V$ and $\RRe \gota(\widetilde{P} u, u- \widetilde{P} u) \ge 0$ for all $u \in V$. 
\end{tabeleq}
If $\gota$ is accretive, then the previous conditions are equivalent to 
\begin{tabeleq}
\setcounter{tellerr}{2}
\item \label{invC-Ou-3}
$\widetilde{P}V \subset V$ and $\RRe \gota(u, u- \widetilde{P} u) \ge 0$ for all $u \in V$.
\end{tabeleq}
\end{thm}

The theorem is stated and proved in \cite{Ouh95} or \cite{Ouh5} for accretive forms but the proof given there for the equivalence of \ref{invC-Ou-1} and 
\ref{invC-Ou-2} does not use accretivity. On the other hand, the implication \ref{invC-Ou-2} $\Rightarrow$ \ref{invC-Ou-1} is proved in \cite{ADO} Theorem~2.2
in a general setting of non-autonomous quasi-coercive forms with a non-homogeneous term.  

\medskip
Let now $H$ be a Hilbert space and $j \colon V \to H$ a bounded linear map with dense range.
We assume that $\gota$ is $j$-elliptic and denote by $A$ the operator associated with $(\gota, j)$.
 The semigroup generated by $-A$ on $H$ is denoted by 
$S = (S(t))_{t\ge0}$.

We consider a non-empty closed convex set ${\mathcal C}$ of $H$ and denote by $P \colon  H \to {\mathcal C}$ the projection.
 In the context of $j$-elliptic forms, the previous theorem has the 
following reformulation (see  \cite{AE2}, Proposition~2.9). 

\begin{prop}\label{invC-AE} 
Suppose that $\gota$ is accretive and $j$-elliptic.
Then the following conditions are equivalent. 
\begin{tabeleq}
\item 
$\mathcal C$ is invariant for $S$.
\item 
For all $u \in V$ there exists a $w \in V$ such that $P(j(u))= j(w)$ and $\RRe \gota(w, u-w) \ge 0$.
\item 
For all $u \in V$ there exists a $w \in V$ such that $P(j(u))= j(w)$ and $\RRe \gota(u, u-w) \ge 0$.
\end{tabeleq}
\end{prop}

The following invariance criterion is implicit in \cite{AE2}.
It allows to obtain invariance of a closed convex set $\mathcal C$ in $H$ 
for the semigroup $S$ from  the invariance of a certain closed convex set 
$\widetilde{\mathcal C}$ for the semigroup $\widetilde S$ in $\widetilde{H}$.

\begin{prop}\label{invC} 
Assume $\gota$ is accretive and $j$-elliptic.
Suppose that the convex set $\widetilde{\mathcal C}$ is invariant for the semigroup $\widetilde S$ and that  
\begin{equation}\label{eq3.1}
P \circ j = j \circ \widetilde{P}.
\end{equation}
Then the convex set ${\mathcal C}$ is invariant for the  semigroup $S$.  
\end{prop}
\begin{proof} First,  note that the  term in the right hand side of condition 
(\ref{eq3.1}) makes sense because of the fact that $\widetilde{P} V \subset V$ by 
Theorem~\ref{invC-Ou} and $j \colon  V \to H$. 

 Let now $u \in V$ and define $w = \widetilde{P}u$.
Then $w \in V$ and $Pj(u) = j(\widetilde{P} u) = j(w)$.
Moreover, 
\[
\RRe \gota(w, u-w) = \RRe \gota (\widetilde{P} u, u- \widetilde{P} u)  \ge 0
\]
 by Theorem~\ref{invC-Ou} and the assumption that $\widetilde{\mathcal C}$ is invariant for the semigroup $\widetilde S$.
We conclude by Proposition~\ref{invC-AE} that ${\mathcal C}$ is invariant for $S$.
 \end{proof}
 
There are interesting situations where one would like to relax  the  
accretivity assumption in the previous results.  
A typical situation is when one applies the above criteria to positivity of the Dirichlet-to Neumann semigroup. 
For example, if one considers the form given by (\ref{pos01}) with $a_0 = \lambda \in \Ri$, then the accretivity (on $W^{1,2}(\Omega)$) holds only if $\lambda \geq 0$. 
The accretivity on $W^{1,2}_0(\Omega)$, however,
holds if $\lambda \geq - \lambda_0$, where $\lambda_0$ is the first (positive) eigenvalue of the elliptic operator given in (\ref{eq-ellip01}) with Dirichlet boundary conditions. 
 It is then of interest to know whether one can replace accretivity in the previous results by accretivity on $W^{1,2}_0(\Omega)$. 
In the light of Theorem \ref{invC-Ou}, one would expect to have equivalence of (i) and (ii) in Proposition~\ref{invC-AE} in general.  
It turns out that this is true if the form $\gota$ is symmetric. We do not know whether the  same result holds in the case of non-symmetric forms.

Before stating the results we need some notation and assumptions. 
Set 
\[ 
V(\gota) = \{ u \in V : \gota(u,v) = 0  \mbox{ for all } v \in \ker j \}.
\]
In Example~\ref{xdia201} the space $V(\gota)$ coincides with the space of ${\mathcal A}$-harmonic functions. We assume that 
\begin{equation}\label{hyp-sum}
V = V(\gota) \oplus \ker j
\end{equation}
as vector spaces. 
In addition, we assume that there exist $\omega \in \Ri$ and $\mu > 0$ such that  
\begin{equation} \label{j-ell-V(a)}
 \RRe \gota(u,u) + \omega \, \|j(u) \|_H^2 \ge \mu \,  \| u \|_V^2
 \end{equation}
for all $u \in V(\gota)$.
(Loosely speaking, the $j$-ellipticity holds only on $V(\gota)$.)

Under these two assumptions, one can define as previously the operator $A$ associated with $(\gota, j)$ 
and $A$ is m-sectorial
(see \cite{AE2} Corollary 2.2). 
We denote again by $S$ the semigroup generated by $-A$ on $H$. The we have the following version of Proposition \ref{invC-AE}  in which we relax the accretivity assumption to be valid only
on $\ker j$. 
Note that we always assume that $j \colon V \to H$ is continuous and has dense range. 

\begin{prop} \label{invC-AE-New} 
Suppose that the form $\gota$ is symmetric and satisfies (\ref{hyp-sum}) and  (\ref{j-ell-V(a)}).  
Suppose in addition that 
\[
\gota(u,u) \ge 0 
\quad \mbox{for all } u \in \ker j
.  
\]
Then the following conditions are equivalent. 
\begin{tabeleq}
\item \label{invC-AE-New-1} 
$\mathcal C$ is invariant for $S$.
\item \label{invC-AE-New-2} 
For all $u \in V$ there exists a $w \in V$ such that $P j(u) = j(w)$ and $\RRe \gota(w, u-w) \ge 0$.
\end{tabeleq}
\end{prop}

\begin{remark}
The implication \ref{invC-AE-New-1}$\Rightarrow$\ref{invC-AE-New-2}
remains valid without the symmetry assumption of the form $\gota$. 
\end{remark}

\begin{proof}[{\bf Proof of Proposition~\ref{invC-AE-New}}.]
Define  $\gota_c \colon j(V(\gota)) \times j(V(\gota)) \to \Ci$ by
\[
\gota_c(j(u),j(v)) := \gota(u,v)
\]
for all $u,v \in V(\gota)$.
We provide $j(V(\gota))$ with the norm carried over from $V(\gota)$ by $j$ using~(\ref{hyp-sum}).
It is easy to see that the form $\gota_c$ is well defined, continuous and quasi-coercive. 
Its associated operator is again $A$ (see \cite{AE2} Theorem 2.5 and  one can easily replace the $j$-ellipticity there by (\ref{j-ell-V(a)})). 
Now we can apply Theorem~\ref{invC-Ou} in which the equivalence of the first two assertions does not use accretivity.

`\ref{invC-AE-New-1}$\Rightarrow$\ref{invC-AE-New-2}'. 
By Theorem~\ref{invC-Ou} we have 
$P(j(V(\gota))) \subset j(V(\gota))$. 
Let  $u \in V$.
By (\ref{hyp-sum}) there exists a $u' \in V(\gota)$ such that $j(u) = j(u')$.
Hence there is a $w \in V(\gota)$ such that $P j(u') = j(w)$.
Then $P j(u) = P j(u') = j(w)$.
In addition, since $u-u' \in \ker j$ and $w \in V(\gota)$, we have 
\begin{eqnarray*}
\RRe \gota(w, u-w) &=& \RRe \gota(w, u-u') + \RRe \gota(w, u'-w)\\
&=& \RRe \gota(w, u'-w)\\
&=& \RRe \gota_c (j(w), j(u' - w))\\
&=&  \RRe \gota_c (P j(u'), j(u')- Pj(u'))\\
&\ge&  0,
\end{eqnarray*}
where we use again Theorem~\ref{invC-Ou} in the last step. 
This gives Condition~\ref{invC-AE-New-2}. 
We observe that the symmetry assumption is not used here.

`\ref{invC-AE-New-2}$\Rightarrow$\ref{invC-AE-New-1}'.
Let $\varphi := j(u) \in D(\gota_c)$, where $u \in V(\gota)$. 
By \ref{invC-AE-New-2} there exists a $w \in V$ such that $P j(u) = j(w)$ and 
$\RRe \gota(w, u-w) \geq 0$.
By (\ref{hyp-sum}) there is a $w' \in V(\gota)$ such that $j(w) = j(w')$.
Then $P\varphi = Pj(u) = j(w) = j(w') \in D(\gota_c)$.
Next 
\begin{eqnarray*}
\RRe \gota_c(P \varphi, \varphi - P \varphi) &=& \RRe \gota_c (j(w'), j(u) - j(w'))\\
&=& \RRe \gota(w', u-w')\\
&=& \RRe \gota(w'-w, u-w') + \RRe \gota(w, u-w')\\
&=& \RRe \gota(w, u-w').
\end{eqnarray*}
Here we use 
$$\RRe \gota(w'-w, u-w') = \RRe \gota(u-w', w'-w) = 0$$
by the symmetry of $\gota$ and the facts that $u-w' \in V(\gota)$ and $w'-w \in \ker j$. 
Now, by Condition~\ref{invC-AE-New-2} one deduces that 
\begin{eqnarray*}
\RRe \gota(w, u-w') 
&=& \RRe \gota(w, u-w) + \RRe \gota(w, w-w')\\
&\ge& \RRe \gota(w, w-w').
\end{eqnarray*}
On the other hand, $\RRe \gota(w', w-w') = 0$ since  $w' \in V(\gota)$ and $w-w' \in \ker j$.
Therefore 
\begin{eqnarray*}
\RRe \gota(w, w-w') &=&  \RRe \gota(w-w', w-w') + \RRe \gota(w', w-w')\\
&=& \RRe \gota(w-w', w-w') \\
&\ge& 0 ,
\end{eqnarray*}
where we use the accretivity assumption on $\ker j$. 
Hence we proved that
\[
\RRe \gota_c(P \varphi, \varphi - P \varphi) \ge 0
.  \]
Using again Theorem~\ref{invC-Ou}\ref{invC-Ou-2}$\Rightarrow$\ref{invC-Ou-1} 
we conclude that $\mathcal C$ is invariant for $S$.
\end{proof}

Now we have the following version of Proposition \ref{invC} with an identical proof, except that now we apply  Proposition \ref{invC-AE-New} instead of 
Proposition \ref{invC-AE}. 

\begin{cor}\label{invC-New} 
Assume that the form $\gota$ is symmetric and satisfies (\ref{hyp-sum}) and  (\ref{j-ell-V(a)}).  
Suppose in addition that 
$\gota(u,u) \ge 0$ for all $ u \in \ker j$
Suppose that the convex set $\widetilde{\mathcal C}$ is invariant for the semigroup $\widetilde S$ and that  
\[
P \circ j = j \circ \widetilde{P}.
\]
Then the convex set ${\mathcal C}$ is invariant for the  semigroup $S$.  
\end{cor}

We conclude this section by mentioning that one may consider the 
Condition~\ref{invC-AE-New-2} in Theorem~\ref{invC-Ou}, Proposition~\ref{invC-AE} and 
Proposition~\ref{invC-AE-New} on a dense subset of $V$ as in 
\cite{Ouh5} Theorem~2.2.

\section{Positivity and $L_\infty$-contractivity}\label{pos}

The criteria in the previous section turn out to be simple and effective in applications.
We illustrate this by proving positivity and $L_\infty$-contractivity of the 
semigroup generated by the Dirichlet-to-Neumann operator $\cn$ described in 
Example~\ref{xdia201} of Section~\ref{Stratford} under a mild additional condition. 
This mild condition is that there is a $\mu > 0$ such that 
\begin{equation}\label{eq3.2}
\RRe \gota(u, u) 
\ge \mu \, \| \nabla u \|_{L_2(\Omega)}^2 
\end{equation}
for all $u \in W^{1,2}(\Omega)$.
This condition is valid if $\RRe a_0$ is large enough. 
It is a standard fact that there is a $\mu' > 0$ such that 
\[
\int_\Omega | \nabla u |^2 + \int_\Gamma | \Tr(u) |^2 
\ge \mu' \, \| u \|_{W^{1,2}(\Omega)}^2 
\]
for all $u \in W^{1,2}(\Omega)$.
 From this and (\ref{eq3.2}), it follows that $\gota$ is $j$-elliptic with $j = \Tr$.
Then we have the following consequence of Proposition~\ref{invC}.

\begin{cor}\label{cor3.4}
Suppose  (\ref{eq3.2}) and 
that the coefficients $c_{kl}, b_k, c_k$ and $a_0$ are all real-valued for all 
$k,l \in \{ 1,\ldots,d \} $.
Then the semigroup $S$ generated by (minus) the Dirichlet-to-Neumann operator $\cn$ is positive.
\end{cor}
\begin{proof} 
It follows from \cite{Ouh5}, Theorem 4.2, that the semigroup $\widetilde S$ generated by $-A$ on $L_2(\Omega)$ is positive.
Therefore $\widetilde S$ leaves invariant the closed convex set 
$\widetilde{\mathcal C} := \{ u \in L_2(\Omega) : u \ge 0 \}$.
The projection onto $\widetilde{\mathcal C} $ is 
$\widetilde{P} u = (\RRe u)^+$.
Now we choose ${\mathcal C} := \{ \varphi  \in L_2(\Gamma) : \varphi \ge 0 \}$.
Then $P \varphi = (\RRe \varphi)^+$. 
It is clear that (\ref{eq3.1}) is satisfied and hence ${\mathcal C}$ is invariant for $S$ by Proposition~\ref{invC}.
This latter property means that $S$ is positive.
 \end{proof}

Regarding the positivity proved above, a remark is in order. 
We have assumed (\ref{eq3.2}) in order to ensure $j$-ellipticity and define the 
Dirichlet-to-Neumann operator using $(\gota, j)$ technique as explained in Section \ref{Stratford}. 
The condition (\ref{eq3.2}) is however not true for general (too negative) $a_0$.  
On the other hand, for general $a_0 \in L_\infty(\Omega)$ one can still define the 
Dirichlet-to-Neumann operator using the form (\ref{eq-b}) under the sole condition 
that the elliptic operator with Dirichlet boundary conditions is invertible on $
L_2(\Omega)$. 
If $\gota$ is symmetric, then we apply Corollary~\ref{invC-New} instead of 
Proposition~\ref{invC} and obtain that the Dirichlet-to-Neumann semigroup $S$ is 
positive if in addition the form $\gota$ is accretive on $W^{1,2}_0(\Omega)$. 
In particular, if $c_{kl} = c_{lk}$ and $b_k = c_k$ for all
$k,l \in \{ 1,\ldots,d \} $, then $S$ is positive as soon as $a_0 + \lambda_1^D > 0$
(that is $a_0 + \lambda_1^D \geq 0$ and 
not $a_0 + \lambda_1^D = 0$ almost everywhere),
where $\lambda_1^D$ is the first eigenvalue of the operator 
\[
- \sum_{k,l=1}^d \partial_l \left( c_{kl} \, \partial_k \right) 
   + \sum_{k=1}^d b_k \, \partial_k -\partial_k (c_k \cdot)
\]
subject to the Dirichlet boundary conditions. 
Note that if the condition $a_0 + \lambda_1^D > 0$ is not satisfied, 
the semigroup $S$ might not be positive. 
See  \cite{Dan}.

\medskip

Concerning the $L_\infty$-contractivity of the Dirichlet-to-Neumann semigroup $S$ we have the following result. 
\begin{cor}\label{cor3.5} 
Suppose in addition to (\ref{eq3.2})  that $\RRe a_0 \ge 0$. 
Suppose also that $c_{kl}, b_k$ and $i c_k$ are real-valued for all 
$k,l \in \{ 1,\ldots,d \} $.
Then the semigroup $S$ is $L_\infty$-contractive.
\end{cor}
\begin{proof} 
Under the assumptions of the corollary, the semigroup $\widetilde S$ is 
$L_\infty$-contractive by Theorem 4.6  in \cite{Ouh5}.
This means that $\widetilde S$ leaves invariant the closed convex set given by 
$\widetilde{\mathcal C} := \{ u \in L_2(\Omega) : | u |  \le 1 \}$.
The projection onto $\widetilde{\mathcal C} $ is 
$\widetilde{P} u = (1 \land |u |) \sgn u$.
We choose ${\mathcal C} := \{ \varphi  \in L_2(\Gamma) : | \varphi | \le 1 \}$.
Then $P \varphi = (1 \land |\varphi |) \sgn \varphi$.
Since $\Tr ( (1 \land |u |) \sgn u ) = (1 \land | \Tr u |) \sgn (\Tr u)$ the condition 
(\ref{eq3.1}) is satisfied and hence ${\mathcal C}$ is invariant for $S$ by Proposition~\ref{invC}.
This proves that $S$ is $L_\infty$-contractive.
\end{proof}

A consequence of the previous corollary is that the semigroup $S$ can be extended to 
a holomorphic semigroup on $L_p(\Gamma)$ for all $p \in (2, \infty)$. 
For all $p \in (1,2)$ one may argue by duality by applying the corollary to the adjoint operator.

\section{A domination criterion}\label{Whangamomona}

This section is devoted to a domination criterion for semigroups such as those generated by 
Dirichlet-to-Neumann operators.
Although one can find in the literature  several criteria for the domination in terms of 
sesquilinear  forms (see \cite{Ouh95} or Chapter 2 in \cite{Ouh5}) their application 
to Dirichlet-to-Neumann operators is difficult since one has to deal with 
harmonic lifting of functions such as $\varphi \sgn \psi$ with 
$\varphi, \psi \in H^{1/2}(\Gamma)$ such that $|\varphi| \le |\psi |$ 
(see Theorem~\ref{dom-Ou} below).
In contrast to general criteria in \cite{Ouh95} we shall focus on operators such as 
the Dirichlet-to-Neumann operator and make a link between the domination in $L_2(\Gamma)$ 
and the domination  in $L_2(\Omega)$.
In a sense, we obtain the domination in $L_2(\Gamma)$ for the semigroup generated by
(minus) the Dirichlet-to-Neumann operator from  the domination  in $L_2(\Omega)$ of the 
corresponding elliptic operator with Neumann boundary conditions. 

We start by fixing some notation.
Let $\widetilde{H} := L_2(\widetilde X, \tilde \nu)$ and $H = L_2(X, \nu)$, 
where $(\widetilde X, \tilde \nu)$ and $(X, \nu)$ are $\sigma$-finite measure spaces.
Let $U$ and $V$ be two Hilbert spaces which are densely and continuously 
embedded into $\widetilde{H}$.
We consider two sesquilinear forms
\[
\gota \colon U \times U \to \Ci
\quad \mbox{and} \quad
\gotb \colon V \times V \to \Ci
\]
which are continuous, accretive and quasi-coercive.
We denote by $\widetilde{A}$ and $\widetilde{B}$ their associated operators, respectively.
Let $j_1 \colon U \to H$ and $j_2 \colon V \to H$ be two bounded operators with dense ranges.
We assume that $\gota$ is $j_1$-elliptic and $\gotb$ is $j_2$-elliptic and denote by 
$A$ and $B$ the operators associated with 
$(\gota, j_1)$ and $(\gotb, j_2)$, respectively.
Finally, we denote by $\widetilde{S} = (\widetilde{S}(t))_{t\ge0}$ and 
$\widetilde{T} = (\widetilde{T}(t))_{t\ge0}$ the semigroups generated by 
$-\widetilde{A}$ and $-\widetilde{B}$ on $\widetilde{H}$ and $S = (S(t))_{t\ge0}$ and $T = (T(t))_{t\ge0}$ the semigroups generated by $-A$ and $-B$ on $H$, respectively.
Then under these assumptions we have transference of domination.

\begin{thm}\label{dom}
Adopt the above notation and assumptions.
Further suppose the following.
\begin{tabelR}
\item \label{dom-1}
$\widetilde{S}$ is dominated by $\widetilde{T}$, i.e., 
\[
| \widetilde{S}(t) f | \le \widetilde{T}(t) |f|
\]
for all $t \ge 0$ and $f \in \widetilde{H}$. 
\item \label{dom-2}
The maps $j_1$ and $j_2$ satisfy the four properties in Hypothesis~\ref{hypdom}.
\end{tabelR}
Then $S$ is dominated by $T$, i.e.,
\[
| S(t) \varphi | \le T(t) | \varphi |
\]
for all $t \ge 0$ and all $\varphi \in H$. 
\end{thm}

In the light of Proposition \ref{invC-AE-New} and Corollary \ref{invC-New} the accretivity assumption can be improved if 
 the forms $\gota$ and $\gotb$ are symmetric. We leave the details to the interested reader. 

For the proof of Theorem~\ref{dom} and also to formulate Hypothesis~\ref{hypdom}
we need some additional concepts.
The following definition was introduced in \cite{Ouh95}.

\begin{ddefinition} 
We say that $U$ is an {\bf ideal} of $V$ if 
\begin{itemize}
\item $u \in U  \Rightarrow | u | \in V$ and
\item if $u \in U$ and $v \in V$ are such that $|v| \le |u|$, then $v \sgn u \in U$. 
\end{itemize}
\end{ddefinition}

We also recall the following criterion for domination (see \cite{Ouh95} or \cite{Ouh5} Theorem~2.21).

\begin{thm}\label{dom-Ou} 
Suppose that the semigroup $\widetilde{T}$ is positive.
The following conditions are equivalent.
\begin{tabeleq}
\item
$\widetilde{S}$ is dominated by $\widetilde{T}$.
\item
$U$ is an ideal of $V$ and $\RRe \gota(u, |v| \sgn u) \ge \gotb(|u|, |v|)$ 
for all $(u,v) \in U \times V$ such that $|v| \le |u|$.
\item
$U$ is an ideal of $V$ and $\RRe \gota(u, v) \ge \gotb(|u|, |v|)$ 
for all $u,v \in U$ such that $u\, \overline{v} \ge 0$.
\end{tabeleq}
\end{thm}
 
Since we assume in Theorem~\ref{dom} that $\widetilde{S}$ is dominated by $\widetilde{T}$,
it is then is a consequence of Theorem~\ref{dom-Ou} that $U$ is an ideal of  $V$.
In particular, all the quantities appearing in the  following properties   are well defined.

\begin{hyp} \label{hypdom}
Assume
\begin{itemize}
\item
$ j_2(\RRe v) = \RRe j_2(v)$ for all $v \in V$, 
\item
$ j_2 (v_1 \lor v_2) = j_2(v_1) \lor j_2(v_2)$ for all $v_1, v_2 \in V$ which are real-valued,
\item
$  j_2( |u|) = | j_1(u)|$ for all $u \in U$, and 
\item
$  j_1(|v| \sgn u) = |j_2(v)| \sgn(j_1(u))$ for all $(u,v) \in U\times V$ such that $|v| \le |u|$.
\end{itemize}
\end{hyp}

Note that the first two properties use the fact that semigroup $\widetilde{T}$ is 
positive and hence $\RRe u , (\RRe u )^+ \in V$ for all $u \in V$.
This implies that 
$v_1 \lor v_2 \in V$ for all real-valued $v_1, v_2 \in V$. 

\medskip

Obviously, the properties in Hypothesis~\ref{hypdom} are satisfied if 
$U = V = W^{1,2}(\Omega)$, $H = L_2(\Gamma)$ and $j_1 = j_2 = \Tr$. 

\begin{proof}[{\bf Proof of Theorem~\ref{dom}}] 
We follow an idea from \cite{Ouh95} and view the domination as the  
invariance of a closed convex set by an appropriate semigroup.
Define $\widehat{H} := \widetilde{H} \times  \widetilde{H} = L_2(\widetilde X, \mu) \times L_2(\widetilde X,\mu)$ 
and consider the closed convex set
\[
\widehat{\mathcal C} := \{ (u,v) \in \widehat{H} : |u| \le v \}.
\]
The projection onto $\widehat{\mathcal C}$ is given by 
\begin{equation}\label{eq4.4}
\widehat{P} (u,v) 
= \frac{1}{2} \left( \left[ |u| + |u|\land \RRe v \right]^+ \sgn u, 
       \left[ |u| \lor \RRe v + \RRe v \right]^+ \right).
\end{equation}
See \cite{Ouh95} or \cite{Ouh5} (2.7).
We also define $\hat{j} \colon U \times V \to H \times H$ by 
\[
\hat{j}(u,v) := (j_1(u), j_2(v)).
\]
Since $j_1$ and $j_2$ are bounded with dense ranges it is clear that $\hat{j}$ is 
bounded and has dense range. 

Next define the sesquilinear form $ \gotc \colon (U \times V) \times (U\times V) \to \Ci$ by
\[
\gotc((u_0,v_0), (u_1,v_1)) := \gota(u_0,u_1) + \gotb(v_0,v_1).
\]
This form is quasi-coercive, accretive and continuous.
Its associated operator is 
\[
\left( { \begin{array}{cc}
\widetilde{A} & 0 \\
0 & \widetilde{B}\\
\end{array} } \right)
\]
and the corresponding semigroup on $\widehat{H}$ is
\[ \left( { \begin{array}{cc}
\widetilde{S} & 0 \\
0 & \widetilde{T}\\
\end{array} } \right) = \left( { \begin{array}{cc}
\widetilde{S}(t) & 0 \\
0 & \widetilde{T}(t)\\
\end{array} } \right)_{t\ge0}.
\]
We next show that $\gotc$ is $\hat{j}$-elliptic.
Indeed, if $(u,v) \in U \times V$, then
\begin{eqnarray*}
\RRe \gotc((u,v), (u,v)) + \omega \, \| \hat{j} (u,v) \|_{H\times H}^2 
&= &\RRe \gota(u,u) + \omega \, \| j_1(u) \|_H^2 + \RRe \gotb(v,v) + \omega \, \| j_2(v) \|_H^2\\
&\ge & \mu \, ( \| u \|_U^2 + \| v\|_V^2),
\end{eqnarray*}
where we use that $\gota$ is $j_1$-elliptic and $\gotb$ is $j_2$-elliptic 
with some constants $\omega_1, \omega_2 \in \Ri$ and $\mu_1, \mu_2 > 0$ and then we take 
$\omega = \max(\omega_1, \omega_2)$ and $\mu = \min(\mu_1, \mu_2)$. 
Recall that $A$ and $B$ are the operators associated with 
$(\gota, j_1)$ and $(\gotb,j_2)$, respectively.
Denote by $C$ the operator associated with
$(\gotc, \hat{j})$.

We shall show that 
\begin{equation}\label{eq4.6}
C = \left( { \begin{array}{cc}
A & 0 \\
0 & B\\
\end{array} } \right). 
\end{equation}
In order to prove this we use the definition of the associated operator.
Let $(\varphi, \psi) \in D(C)$ and write $(\eta, \chi) = C (\varphi, \psi)$.
This means that there exists 
$(u,v) \in U \times V$ such that 
\begin{eqnarray}
\hat{j}(u,v) &=& (\varphi, \psi) \mbox{ and} \label{eq4.7}\\
\gotc((u,v) , (w,z) ) &=& ( (\eta, \chi), \hat{j}(w,z) )_{H \times H} \mbox{ for all } (w,z) \in U\times V.
\label{eq4.8}
\end{eqnarray}
The equality in (\ref{eq4.8}) reads as
\[
\gota(u,w) + \gotb(v,z) = (\eta, j_1(w))_H + (\chi, j_2(z))_H
\]
for all $w \in U$ and $z \in V$.
Taking $z= 0$ in the last equality and using (\ref{eq4.7}) yields
$\varphi = j_1(u)$ and $\gota(u,w) = (\eta, j_1(w))_H$ for all $w \in U$.
This means that  $\varphi \in D(A)$ and $A \varphi = \eta$.
Similarly, $\psi \in D(B)$ and $B \psi = \chi$. 
Hence
\[
(\varphi, \psi) \in D( \left( { \begin{array}{cc}
A & 0 \\
0 & B
\end{array} } \right) )
\quad \mbox{and} \quad  C(\varphi, \psi) = \left( { \begin{array}{cc}
A & 0 \\
0 & B
\end{array} } \right)(\varphi, \psi).
\]
We have proved that $\left( { \begin{array}{cc}
A & 0 \\
0 & B
\end{array} } \right)$ is an extension of $C$.
The converse inclusion is similar and we obtain (\ref{eq4.6}).

We conclude from the equality (\ref{eq4.6}) that the semigroup generated by 
$-C$ is given by 
\[ \left( { \begin{array}{cc}
S & 0 \\
0 & T
\end{array} } \right) = \left( { \begin{array}{cc}
S(t) & 0 \\
0 & T(t)
\end{array} } \right)_{t\ge0}.
\]
Now we consider the closed convex subset of $H\times H$ defined by
\[ 
{\mathcal C} := \{ (\varphi, \psi) \in H\times H : | \varphi| \le \psi \}.
\]
Similarly to (\ref{eq4.4}), the projection onto ${\mathcal C}$ is given by
\[
P (\varphi,\psi) 
= \frac{1}{2} \left( \left[ |\varphi| + |\varphi| \land \RRe \psi \right]^+ \sgn \varphi, \left[ |\varphi| \lor \RRe \psi + \RRe \psi \right]^+ \right).
\]
It follows easily from Hypothesis~\ref{hypdom} that $P\circ \hat{j} = \hat{j} \circ \widehat{P}$.
Since the domination of $\widetilde{S}$ by $\widetilde{T}$ means that the semigroup $\left( { \begin{array}{cc}
\widetilde{S} & 0 \\
0 & \widetilde{T}
\end{array} } \right)  
$ leaves invariant the convex $\widehat{\mathcal C}$, we conclude by Proposition~\ref{invC} that the semigroup 
$\left( { \begin{array}{cc}
S & 0 \\
0 & T
\end{array} } \right) $, generated by $-C$ on $H \times H$, leaves invariant the convex set ${\mathcal C}$.
The latter property means again that 
$S$ is dominated by $T$.
This proves the theorem.
\end{proof}

\section{The diamagnetic inequality}\label{Karangahake}

In this section we prove the diamagnetic inequality for the Dirichlet-to-Neumann operator.
This will be obtained by applying Theorem~\ref{dom}.

Let $\Omega$ be a bounded Lipschitz domain in $\Ri^d$ with  boundary $\Gamma$.
Let $\vec{a} = (a_1,\ldots, a_d)$ be such that $a_k \in L_\infty(\Omega, \Ri)$ 
for all $k \in \{ 1,\ldots, d \} $. 
We consider the  magnetic Dirichlet-to-Neumann operator  $\cn(\vec{a})$ on $L_2(\Gamma)$ and the 
Dirichlet-to-Neumann operator $\cn$ corresponding to $\vec{a} = 0$ 
(see  Examples~\ref{xdia201} and \ref{xdia202} in Section~\ref{Stratford}).
We denote by $S_{\vec{a}} = (S_{\vec{a}}(t))_{t\ge0}$ and $S = (S(t))_{t\ge0}$ 
the semigroups generated by $-\cn(\vec{a})$ and $-\cn$ on $L_2(\Gamma)$, respectively.
We have the following domination.

\begin{thm}\label{diam} 
Let $\Omega$ be a bounded Lipschitz domain in $\Ri^d$ with  boundary $\Gamma$.
Further let $c_{kl}, b_k, c_k, a_0 , a_k \in L_\infty(\Ri)$ 
for all $k,l \in \{ 1,\ldots,d \} $. 
Suppose that the ellipticity condition~(\ref{exdia201;5}) is valid.
Define the form $\gota$ as in (\ref{example1}).
Suppose in addition that the form $\gota$ in (\ref{example1}) is accretive and $j$-elliptic with $j= \Tr$.
Then $S_{\vec{a}}$ is dominated by $S$ on $L_2(\Gamma)$.
That is,
\[ 
| S_{\vec{a}}(t) \varphi | \le S(t) |\varphi|
\]
for all $t \ge 0$ and $\varphi \in L_2(\Gamma)$.
\end{thm}

\begin{proof}
We apply Theorem~\ref{dom} with  $\widetilde{H} = L_2(\Omega)$, 
$U = V = W^{1,2}(\Omega)$ and $ H = L_2(\Gamma)$.
Set $j_1=j_2 = \Tr $.
It is clear that the four properties in Hypothesis~\ref{hypdom} are satisfied.
Therefore Theorem~\ref{diam} follows immediately from Theorem~\ref{dom} and 
the next result, Proposition~\ref{diamOm}, on the domination in $L_2(\Omega)$.
\end{proof}

Denote by ${\widetilde A}(\vec{a})$ and $\widetilde{A} = \widetilde{A}(0)$ 
the elliptic operators associated with the forms defined by (\ref{example2}) 
and (\ref{example1}) on $W^{1,2}(\Omega)$.
We denote by $\widetilde{S_{\vec{a}}}$ and $\widetilde{S}$ the semigroups 
generated by $- \widetilde{A}(\vec{a})$ and $-\widetilde{A} $ on $L_2(\Omega)$, respectively.

\begin{prop}\label{diamOm} 
Suppose that $c_{kl}, b_k, c_k, a_0$ and $a_k$ 
are real-valued for all $k,l \in \{ 1,\ldots,d \} $. 
Then we have the diamagnetic inequality
\[ 
| \widetilde{S_{\vec{a}}}(t) f | \le \widetilde{S}(t) |f|
\]
for all $t \ge 0$ and $f \in L_2(\Omega)$.
\end{prop}

The proposition is very well known in the case $\Omega = \Ri^d$,
$c_{kl} = \delta_{kl}$ and $b_k = c_k = 0$.
For general domains with Neumann boundary conditions (as we do in the previous proposition) 
and $c_{kl} = \delta_{kl}$, $b_k = c_k = 0$ it was proved in \cite{HS04}.
Note that in our case we do not assume any regularity nor symmetry for  $(c_{kl})$.
In addition we allow the presence of terms of order $1$.
The same domination result is also valid, with the same proof,  
if the operators $\widetilde{A}(\vec{a})$ and $\widetilde{A}$ are endowed with 
other boundary conditions such Dirichlet or mixed boundary conditions. 

\begin{proof}  
Note first that since all the coefficients are real-valued, the semigroup 
$\widetilde{S}$ generated by $-\widetilde{A}$ is positive 
(cf.\ \cite{Ouh5} Corollary~4.3).
In particular, $W^{1,2}(\Omega)$ is an ideal of itself (see  \cite{Ouh95} or \cite{Ouh5} Proposition~2.20).
It remains to prove that 
\begin{equation}\label{eq5.1}
\RRe \gota(\vec{a})(u,v) \ge \gota( |u|, |v|)
\end{equation} 
for all $u, v \in W^{1,2}(\Omega)$ with $u\, \overline{v} \ge 0$
and then apply Theorem~\ref{dom-Ou}. 
Let $u, v \in W^{1,2}(\Omega)$ with $u\, \overline{v} \ge 0$.
Then $u\, \overline{v} = |u| \, |v| $ and $ (\sgn  \overline u) \, \sgn v = 1$ outside the sets where 
$ u = 0 $ or $v=0$.
Hence using \cite{GT} Lemma~7.7 one deduces that 
\begin{eqnarray*}
\RRe \gota(\vec{a})(u,v) 
&=& \RRe \sum_{k,l=1}^d \int_\Omega c_{kl} \, (\partial_l u) \, \overline{\partial_k v} 
  + \sum_{k,l=1}^d \int_\Omega c_{kl} \, a_l \IIm(u\, \overline{\partial_k v}) 
  - \sum_{k,l=1}^d \int_\Omega c_{kl} \, a_k \IIm( (\partial_l u) \, \overline{ v}) \\*
&&{}+   \sum_{k,l=1}^d \int_\Omega c_{kl} \, a_l \, a_k \, |u| \, |v| 
   + \sum_{k=1}^d \int_\Omega b_k \RRe ((\partial_k u) \, \overline{v}) 
       + c_k \RRe(u\, \overline{\partial_k v}) \\*
&&{}    + \int_\Omega a_0 \, |u| \, |v| \\
& =&  \sum_{k,l=1}^d \int_\Omega c_{kl} \RRe((\partial_l u) \, \sgn \overline{u}) 
            \RRe ((\partial_k v) \, \sgn \overline{v}) \\*
 &&{} + \sum_{k,l=1}^d \int_\Omega c_{kl} \IIm((\partial_l u) \, \sgn \overline{u}) 
           \IIm ((\partial_k v) \, \sgn \overline{v})  \\*
&&{}+ \sum_{k,l=1}^d \int_\Omega c_{kl} \, a_l \IIm(u\, \overline{\partial_k v}) 
    - \sum_{k,l=1}^d \int_\Omega c_{kl} \, a_k \IIm( (\partial_l u) \, \overline{ v}) \\
 &&{}+  \sum_{k,l=1}^d \int_\Omega c_{kl} \, a_l \, a_k \, |u| \, |v| 
   + \sum_{k=1}^d \int_\Omega b_k \RRe ((\partial_k u)  \, \overline{v}) 
   +c_k \RRe(u\, \overline{\partial_k v})   \\*
& & {}   + \int_\Omega a_0 \, |u| \, |v|\\
&=& \sum_{k,l=1}^d \int_\Omega c_{kl} \, (\partial_l |u|) \, \partial_l |v| 
   + \sum_{k=1}^d \int_\Omega b_k \, (\partial_k |u|) \, |v| 
         +c_k \, |u| \,  \partial_k |v| 
   + \int_\Omega a_0 \, |u| \, |v|\\*
&& {} + \sum_{k,l=1}^d \int_\Omega c_{kl} \IIm((\partial_l u) \, \sgn \overline{u}) 
           \IIm ((\partial_k v) \, \sgn \overline{v}) \\*
&& {} - \sum_{k,l=1}^d\int_\Omega c_{kl} \, a_k \IIm ((\partial_l u) \, \sgn \overline{u}) \, |v| 
   - \sum_{k,l=1}^d \int_\Omega c_{kl} \, a_l \IIm ((\partial_k u) \, \sgn \overline{u}) \, |v|  \\*
&& {} + \sum_{k,l=1}^d \int_\Omega c_{kl} \, a_l \, a_k \, |u| \, |v|,
\end{eqnarray*}
where we used the standard fact that 
\[  
\partial_k |u| = \RRe ((\partial_k u ) \, \sgn \overline{u} ).
\]
In addition, since $u\, \overline{v}  \ge 0$ we have $\IIm \partial_k (u \, \overline{v}) = 0$ 
and hence we used that 
\[  
\IIm( u \, \overline{\partial_k v} )
= - |v| \IIm( (\partial_k u) \, \sgn \overline{u})
.  \]
Next $- |u| \IIm( (\partial_k v) \, \sgn \overline{v})
= \IIm( u \, \overline{\partial_k v} )$ and therefore 
\[
\int_\Omega c_{kl} \IIm((\partial_l u) \, \sgn \overline{u}) \IIm ((\partial_k v) \, \sgn \overline{v})
= \int_\Omega c_{kl} \IIm((\partial_l u) \, \sgn \overline{u}) \IIm ((\partial_k u) \, \sgn \overline{u}) 
         \, \frac{|v|}{|u|},
\]
with the convention that 
$\IIm((\partial_l u) \, \sgn \overline{u}) \IIm ((\partial_k u) \, \sgn \overline{u}) \, \frac{|v|}{|u|} = 0$ 
on the set where $u = 0$.

It follows that 
\begin{eqnarray*}
\RRe \gota(\vec{a})(u,v) 
&=& \gota(|u|, |v|) 
   + \sum_{k,l=1}^d \int_\Omega c_{kl} \IIm((\partial_l u) \, \sgn \overline{u}) 
          \IIm ((\partial_k u) \, \sgn \overline{u}) \, \frac{|v|}{|u|}\\*
&& {} - \sum_{k,l=1}^d \int_\Omega (c_{kl}+c_{lk}) \, a_k \IIm((\partial_l u) \, \sgn \overline{u}) \, |v| 
   +  \sum_{k,l} \int_\Omega c_{kl} \, a_l \, a_k \, |u| \, |v|\\
&=&  \gota(|u|, |v|)  + \int_\Omega Q \, \frac{|v|}{|u|},
\end{eqnarray*}
where 
\begin{eqnarray*}
Q 
&=& \sum_{k,l=1}^d  c_{kl} \IIm((\partial_l u) \, \sgn \overline{u}) 
           \IIm ((\partial_k u) \, \sgn \overline{u})   
- \sum_{k,l=1}^d  (c_{kl}+c_{lk}) \, a_k \IIm((\partial_l u) \, \sgn \overline{u}) \, |u|\\*
&& {} + \sum_{k,l=1}^d c_{kl} \, a_l \, a_k \, |u|^2.
\end{eqnarray*}
It remains to prove that $Q \ge 0$ to obtain (\ref{eq5.1}). 

Set $\xi_k := \IIm((\partial_k u) \, \sgn \overline{u})$ for all $k \in \{ 1,\ldots,d \} $, 
$\xi = (\xi_1,\ldots,\xi_d)$ and $C = (c_{kl})_{1 \le k,l \le d}$.
Then
\[ 
Q 
= \langle C \xi, \xi \rangle_{\Ri^d}  
   -  \langle (C+C^*) \vec{a}, \xi \rangle_{\Ri^d} \, |u| 
   +  \langle C \vec{a}, \vec{a}\rangle_{\Ri^d} \, |u|^2.
\]
By the Cauchy--Schwarz inequality,
\begin{eqnarray*}
 \langle (C+C^*) \vec{a}, \xi \rangle_{\Ri^d} \, |u| 
&\le&  \langle (C+C^*) \vec{a}, \vec{a} \rangle_{\Ri^d}^{1/2} \, |u| \: 
        \langle (C+C^*) \xi, \xi \rangle_{\Ri^d}^{1/2}\\
 &\le& \frac{1}{2} \langle (C+C^*) \vec{a}, \vec{a} \rangle_{\Ri^d} \,  |u|^2 
    + \frac{1}{2}  \langle (C+C^*) \xi, \xi \rangle_{\Ri^d}\\
 &=& \langle C \vec{a}, \vec{a}\rangle_{\Ri^d} \, |u|^2
   + \langle C \xi, \xi \rangle_{\Ri^d} .
 \end{eqnarray*}
 This implies that $Q \ge 0$ and finishes the proof of the proposition.
 \end{proof}

\begin{remark} 
We mentioned above that the diamagnetic inequality of Proposition~\ref{diamOm} 
is valid with other boundary conditions.
Note also that if we add a positive potential $V$ to $a_0$  in the expression of 
$\widetilde{A}(\vec{a})$, then we have the same domination by the semigroup of 
$\widetilde{A}$ (without $V$).
The same domination holds for the corresponding semigroups of the Dirichlet-to-Neumann operators.
A particular case of this result was proved in \cite{EO4} for the 
Dirichlet-to-Neumann operators associated with $-\Delta + V$ and $-\Delta$ on $L_2(\Gamma)$. 
\end{remark}

\section{Some consequences}\label{Muriwai}

Let $\Omega$ be  a bounded open Lipschitz subset of $\Ri^d$ with  boundary $\Gamma$, where $d \geq 2$.
Let $S_{\vec{a}}$ be the semigroup generated by 
(minus) the magnetic Dirichlet-to-Neumann operator $\cn(\vec{a})$ on $L_2(\Gamma)$.
Since the trace operator is compact, it follows that the spectrum of $\cn(\vec{a})$ is discrete.
If $\cn(\vec{a})$ is self-adjoint we denote by $\lambda_1 \le \lambda_2 \le \ldots$ 
the sequence of the corresponding eigenvalues.
The first consequence of Theorem~\ref{diam} is as follows.

\begin{cor}\label{trace}
Suppose that $c_{kl} = c_{lk} \in L_\infty(\Omega, \Ri)$, $ b_k = c_k = 0$ and 
$a_k \in L_\infty(\Omega, \Ri)$ for all $k,l \in \{ 1,\ldots,d \} $.
Suppose also that $a_0 \ge 0$.
Then there exists a constant $c > 0$, independent of $\vec{a}$, such that
\[ 
\sum_{k=1}^\infty e^{-\lambda_k t} \le c \, t^{-(d-1)}
\]
for all $t \in (0,1]$.
\end{cor}
\begin{proof} 
Under the assumptions of the corollary, the operator $\cn(\vec{a})$ is self-adjoint.
In addition, a combination of Theorem~\ref{diam} and Corollary~\ref{cor3.5} shows that 
$S_{\vec{a}}$ is $L_\infty$-contractive.
Now, by the classical Sobolev embeddings the semigroup
$S_{\vec{a}}(t)$ maps $L_2(\Gamma)$ into $L_{\frac{2(d-1)}{d-2}}(\Gamma)$
if $d \geq 3$.
This together with the fact that $S_{\vec{a}}$ is $L_\infty$-contractive 
implies by extrapolation the estimate
\begin{equation}\label{eq6.1}
\| S_{\vec{a}}(t) \|_{L_1(\Gamma) \to L_\infty(\Gamma)} \le c \, t^{-(d-1)}
\end{equation}
for all $t \in (0,1]$.
We refer to \cite{EO4} Theorem~2.6 and \cite{EO6} for additional  details, 
which provide a proof that (\ref{eq6.1}) is also valid if $d = 2$.

The estimate (\ref{eq6.1}) implies that $S_{\vec{a}}(t)$ is given by a kernel 
$K_{\vec{a}}(t,\cdot,\cdot) \colon \Gamma \times \Gamma \to \Ci$ in the sense
\[ 
(S_{\vec{a}}(t) \varphi)(w) = \int_\Gamma K_{\vec{a}}(t,z,w) \, \varphi(z)\, d\sigma(z)
\]
with
\begin{equation}\label{eq6.2}
 | K_{\vec{a}}(t,z,w) | \le c \, t^{-(d-1)}
 \end{equation}
for all $t \in (0,1]$.
 It is well known that the trace of the operator $S_{\vec{a}}(t)$ 
coincides with $\int_\Gamma K_{\vec{a}}(t,z,z) \, d\sigma(z)$ and the corollary 
follows from (\ref{eq6.2}). 
\end{proof}

Note that (\ref{eq6.2}) can also be used to obtain some bounds on the counting function of $\cn(\vec{a})$.
See \cite{AE7}. 

The second consequence we  mention here is that under additional regularity the 
estimate (\ref{eq6.2}) on the heat kernel 
$K_{\vec{a}}$ can be improved into an optimal Poisson bound.

\begin{cor}\label{poisson}
Let  $\Omega$ be a bounded domain  of class $C^{1+ \kappa}$ for some $\kappa > 0$.
Suppose also that $c_{kl} = c_{lk} \in C^\kappa(\Omega, \Ri)$, 
$ b_k = c_k = 0$ and $a_k \in L_\infty(\Omega, \Ri)$ for all $k,l \in \{ 1,\ldots,d \} $.
Suppose in addition that $a_0 \ge 0$.
Then there exists a constant $c > 0$ such that
\[
| K_{\vec{a}}(t,z,w) | 
\leq \frac{c \, (t \wedge 1)^{-(d-1)} \, e^{-\lambda_1 t}}
         {\displaystyle \Big( 1 + \frac{|z-w|}{t} \Big)^d }
\]
for all $z, w \in \Gamma$ and $t > 0$, where $\lambda_1 $ is the 
first eigenvalue of the operator $\cn(\vec{a})$.
\end{cor}
\begin{proof} 
The estimate 
\[
| K_{\vec{a}}(t,z,w) | 
\leq \frac{c \, (t \wedge 1)^{-(d-1)} }
         {\displaystyle \Big( 1 + \frac{|z-w|}{t} \Big)^d }
\]
for all $z, w \in \Gamma$ and $t > 0$ follows immediately from Theorem~\ref{diam} 
and Theorem 1.1 in \cite{EO6}.
The improvement  upon the factor $ e^{-\lambda_1 t}$ follows as at the 
end of the proof of Theorem~1.2 in \cite{EO4} (page~4084).
\end{proof}

\begin{cor} \label{poissonholder}
Adopt the notation and assumptions of Corollary~\ref{poisson}.
In addition suppose that $d \geq 3$.
Then for all $\varepsilon,\tau' \in (0,1)$ and $\tau > 0$ there exist $c,\nu > 0$ 
such that 
\begin{eqnarray*}
\lefteqn{
|K_{\vec{a}}(t,z,w) - K_{\vec{a}}(t,z',w')|
} \hspace*{10mm}  \\*
& \leq & c \, (t \wedge 1)^{-(d-1)} \, 
   \Big( \frac{|z-z'|+|w-w'|}{t + |z-w|} \Big)^\nu \,
    \frac{1}{\displaystyle \Big( 1 + \frac{|z-w|}{t} \Big)^{d-\varepsilon}} \, 
    (1 + t)^\nu \, e^{- \lambda_1 t}
\end{eqnarray*}
for all $z,w,z',w' \in \Gamma$ and $t > 0$ with 
$|z-z'| + |w-w'| \leq \tau \, t + \tau' \, |z-w|$.
\end{cor}
\begin{proof}
This follows by interpolation from the Poisson bounds of Corollary~\ref{poisson} and 
uniform H\"older bounds in \cite{EW1} Theorem~5.5.
The argument is similar to the proof of Theorem~5.11 in \cite{EW1}.
\end{proof}

\subsection*{Acknowledgements} 
This work was carried out when the second named author was visiting the University of
Auckland and the first named author was visiting the University of Bordeaux.
Both authors wish to thank the universities for hospitalities.
The research of A.F.M. ter Elst  is partly supported by the 
Marsden Fund Council from Government funding, 
administered by the Royal Society of New Zealand. 
The research of E.M. Ouhabaz is partly supported by the ANR 
project RAGE,  ANR-18-CE-0012-01.

\end{document}